\documentclass[a4paper,11pt]{amsart}

\usepackage[
textwidth=15cm,
textheight=24cm,
hmarginratio=1:1,
vmarginratio=1:1]{geometry}

\usepackage{graphicx,color,amscd,amsmath,amssymb,verbatim}
\usepackage{amssymb}
\newtheorem{thm}{Theorem}[section]
\newtheorem{lem}[thm]{Lemma}
\theoremstyle{definition}
 \theoremstyle{remark}

\numberwithin{equation}{section}

\newcommand{\D}{\mathbb D}
\newcommand{\C}{\mathbb C}

\theoremstyle{remark}

\numberwithin{equation}{section}
\numberwithin{theorem}{section}

\date{\today}
\begin{document} \title[]{Norm estimates of weighted composition operators pertaining to the Hilbert Matrix }

\author[Lindstr{\"o}m] {Mikael Lindstr\"om}\address{Mikael Lindstr{\"o}m. Department of Mathematics, \AA bo Akademi University. FI-20500 \AA bo, Finland. e.mail: mikael.lindstrom@abo.fi}
\author[Miihkinen] {Santeri Miihkinen} \address{Santeri Miihkinen. Department of Mathematics, \AA bo Akademi University. FI-20500 \AA bo, Finland. e.mail: santeri.miihkinen@abo.fi}
\author[Wikman] {Niklas Wikman}\address{Niklas Wikman. Department of Mathematics, \AA bo Akademi University. FI-20500 \AA bo, Finland. e.mail: niklas.wikman@abo.fi}

\begin{abstract} Very recently,  Bo\v{z}in and Karapetrovi{\'c} \cite{BK} solved a conjecture by proving that the norm of the Hilbert matrix operator  $\mathcal{H}$ on the Bergman space $A^p$ is equal to $\frac{\pi}{\sin(\frac{2\pi}{p})}$ for $2 < p < 4.$ In this article we present a partly new and  simplified proof of this result. 
Moreover, we calculate the exact value of the norm of $\mathcal{H}$ defined on the Korenblum spaces $H^\infty_\alpha$ for 
$0 < \alpha \le 2/3$ and an upper bound for the norm on the scale $2/3 < \alpha < 1$.
\end{abstract}

\maketitle

\section{Introduction}

The Hilbert matrix operator $\mathcal{H}$ is one of the central operators in operator theory and the study of its boundedness, norm and other properties on several analytic function spaces has been under active investigation in recent years.  Magnus \cite{Ma} was the first one to study $\mathcal{H}$ as an operator on the space $l^2$ of all square-summable complex sequences. Diamantopoulus and Siskakis \cite{DS} considered $\mathcal{H}$ acting on the Hardy spaces $H^p$. They showed that $\mathcal{H}\colon H^p \to H^p$ is bounded for $1 < p < \infty,$ and obtained an upper bound for its norm, namely $\|\mathcal{H}\|_{H^p \to H^p} \le \frac{\pi}{\sin(\frac{\pi}{p})}$ for $2 \le p < \infty$ and another estimate on the scale $1 < p < 2$. In \cite{D}, the boundedness of $\mathcal{H}$ on the Bergman spaces $A^p, \, 2 < p < \infty,$ was shown and the first upper bound for the norm of $\mathcal{H}$ 
was also obtained, that is, $$\|\mathcal{H}\|_{A^p \to A^p} \le \frac{\pi}{\sin(\frac{2\pi}{p})}$$ for $4 \le p < \infty$ and another less precise upper estimate for $2 < p < 4$. Later, Dostani{\'c}, Jevti{\'c} and Vukoti{\'c} \cite{DJV} proved this to be the exact value of the norm on the scale $4 \le p < \infty$ and improved the upper bound in \cite{D} for the exponents $2 < p < 4$. They also established the precise value of the norm of $\mathcal{H}$ in the Hardy space case for all $1 < p < \infty$. In turn, the boundedness of $\mathcal{H}$ defined on the Korenblum spaces $H^\infty_\alpha$ was observed in \cite{AMS}. The upper estimate for the norm of $\mathcal{H}$ on $A^p, \, 2 < p < 4,$ was conjectured in \cite{DJV} to be the same as the one for the values $4 \le p < \infty.$ Recently,  Bo\v{z}in and Karapetrovi{\'c} \cite{BK} solved the conjecture in the positive and obtained that $$\|\mathcal{H}\|_{A^p \to A^p} = \frac{\pi}{\sin(\frac{2\pi}{p})}$$ also for $2 < p < 4$.  In \cite{BK}, 
the problem of determining the norm of  $\mathcal{H}$ was reduced to estimates concerning the Beta function.  
The proof of one of their key results  
rests heavily on the use of a theorem of Sturm involving the number of zeros of a polynomial; see \cite{P}. 
The purpose of this article is twofold: First, by avoiding the use of Sturm's theorem, we give a partly new and  simplified proof of the main result in \cite{BK}. 
Secondly, we establish the exact value of the norm of $\mathcal{H}$ defined on the Korenblum spaces $H^\infty_\alpha$ for 
$0 < \alpha < 2/3$ and an upper bound for the norm on the scale $2/3 \le \alpha < 1$. Our proof of the Bergman space case rests upon a new estimate for the Beta function due to Bhayo and S\'andor \cite{BS}. We also give a different proof of case $2+\sqrt{3} \le p < 4$ in \cite[Lemma 2.6]{BK} based on a change of variables in the Beta function. Regarding the $H^\infty_\alpha$ case, we employ the representation of $\mathcal{H}$ in terms of weighted composition operators and a formula for the norm of these operators on spaces $H^\infty_\alpha$.

\section{Preliminaries}
  
  Let $H(\D)$ be the algebra of analytic functions in the open unit disc $\D = \{z \in \C: |z| < 1\}$ of the complex plane $\C$. For those $f \in H(\D), \, f(z) = \sum_{k=0}^\infty a_k z^k,$ for which the sequence $$\left(\sum_{k=0}^\infty \frac{a_k}{n+k+1}\right)_{n=0}^\infty$$ converges, the Hilbert matrix operator $\mathcal{H}$ can be 
  defined by its action on the Taylor coefficients  of $f$ as follows   
\begin{equation*}
\mathcal{H}(f)(z) = \sum_{n = 0}^\infty \left(\sum_{k = 0}^\infty \frac{a_k}{n+k+1}\right) z^n.
\end{equation*}
However, it also admits  an integral representation in terms of weighted composition operators $T_t$
\begin{equation}
\label{eq: weighcomp}
\mathcal{H}(f)(z) = \int_0^1 T_t(f)(z) \, dt, \hspace{0,25cm} z \in \mathbb{D},
\end{equation}
where $T_t(f)(z) = \omega_t(z)f(\phi_t(z))$ for $0  < t < 1,$ $\omega_t(z) = \frac{1}{(t-1)z +1}$ and $\phi_t(z) = \frac{t}{(t-1)z +1}$; see \cite{DS}. The representation \eqref{eq: weighcomp} is employed in this article. We consider $\mathcal{H}$ acting on the Bergman spaces $$A^p = \bigg\{f \in H(\D): \|f\|_{A^p} = \bigg( \int_{\D} |f(z)|^p \, dA(z) \bigg)^{1/p} < \infty \bigg\},$$ where $2 < p < \infty$ and $dA(z)$ is the normalized Lebesgue area measure on $\D$, and on the Korenblum  spaces $$H^\infty_\alpha = \bigg\{f \in H(\D): \|f\|_{H^\infty_\alpha} = \sup_{z \in \D}(1-|z|^2)^\alpha |f(z)| < \infty \bigg\},$$ where $0 < \alpha < 1$. 
We also need the Beta function defined as

\begin{equation*}
	B(s,t) = \int_0^1 x^{s-1}(1-x)^{t-1}\, dx,
\end{equation*}
where $s, t \in \C\setminus \mathbb Z$ satisfy $ \Re(t) > 0$ and $\Re(s) > 0$. It can be checked that
\begin{equation*}
	B(s,t ) = \Gamma(s)\Gamma(t)/\Gamma(s+t),
\end{equation*}
where $\Gamma$ is the Gamma function. We will also use the well-known equation
\begin{equation*}
	\Gamma(z)\Gamma(1-z) = \frac{\pi}{\sin(\pi z)}, \quad z \in \C \setminus \mathbb Z.
\end{equation*}
We refer the reader to \cite{AS} for these and other identities regarding the Beta and Gamma functions.

\section{The norm of the Hilbert matrix operator on $A^p$}

In this section, we present a partly shorter proof than the one given in \cite{BK} for the operator norm of $\mathcal{H}$ acting on the Bergman spaces $A^p, \, 2 < p < 4$. In \cite{BK}, the proof of the main result rests upon a new estimate for the Beta function, see \cite[Lemma 2.6]{BK}, and a part of it is in turn proved by utilizing a classical result of Sturm. 
We are able to avoid the result of Sturm and consequently simplify their proof considerably.  
A crucial tool for us is an estimate for the Beta function proved by Bhayo and S\'andor \cite{BS}, which  
we state and prove next for the convenience of the reader. 

\begin{lem}\label{11} Let $x > 1$, $0 < y < 1$. Then 
\begin{equation}
\label{eq: betaeqn}
B(x,y) < \frac{1}{xy}(x + y - xy).
\end{equation}

\end{lem}
\begin{proof}[Proof.] 
Let us first show that
\begin{equation}
\label{eq: psieq}
		\psi(1+x) - \psi(x+y)  < \frac{1-y}{x+y-xy},
	\end{equation}
	where $\psi(x) = \frac{\Gamma'(x)}{\Gamma(x)}$.
Write 
	\begin{equation*}
		g_x(y) := \psi(1+x) - \psi(x+y) - \frac{1-y}{x+y-xy}.
	\end{equation*}
	For the second derivative of $g_x(y)$ with respect to $y$ notice that
	\begin{equation*}
		\psi''(x+y) = (-1)^3 \sum_{k = 0}^\infty \frac{2!}{{(x+y+k)}^3} < 0,
	\end{equation*}
	see \cite{AS}.
	Thus for $x>1,$ and $0 < y < 1$ we obtain
	\begin{equation*}
		g_x''(y) = \frac{2x-2}{(x(1-y) + y)^3} - \psi''(x+y) > 0.
	\end{equation*}
	It follows now that $g_x$ is convex as a function of $y$. By using the equality $\psi(1+x) - \psi(x) = \frac{1}{x}$ we conclude that
	\begin{equation*}
		g_x(0) = \psi(1+x) - \psi(x) - \frac{1}{x} = 0 = g_x(1).
	\end{equation*}
	By convexity of $g_x$ inequality \eqref{eq: psieq} follows.

Now, we define
	\begin{equation*}
		h_y(x) = \log(\Gamma(1+x)) + \log(\Gamma(1+y)) -\log(\Gamma(x+y)) - \log(x+y-xy).
	\end{equation*}
	Then estimate \eqref{eq: betaeqn} is equivalent to $h_y(x) < 0$.
	By differentiating $h_y$ with respect to $x$ we get
	\begin{equation*}
		h_y'(x) = g_x(y), 
	\end{equation*}
	which is negative due to inequality \eqref{eq: psieq}. Thus $h_y(x)$ is nonincreasing for $x > 1$ and since $h_y(1) = 0$ the result follows.
\end{proof}

Next, we state Lemma 2.5 from \cite{BK} with a significantly simplified proof compared to the one given there. In \cite{BK}, the proof of Lemma 2.5 is divided into three cases depending on the size of the parameter $p.$ Two of the cases utilize a classical result of Sturm  involving the number of zeros of a polynomial; see \cite[Theorem 2.1]{BK}. However, the proof presented here avoids using the result of Sturm  by utilizing Lemma \ref{11} and, moreover, it suffices to consider only two cases, namely $2 < p < 5/2$ and $5/2 \le p < 4$. The proof of the former case is the same as the one in \cite{BK}, but it is presented here for the sake of completeness.

\begin{lem} 
\label{le25}
Let $2 < p < 4$. Then $B\left(\frac{2}{p}, 2(p-2)\right) \leq \frac{1}{(p-2)(4-p)}$.
\end{lem}
\begin{proof}[Proof.]
		\textbf{(1)} Case $2 < p < \frac{5}{2}$. By using $B(x,y) \le \frac{1}{xy}, x,y \in (0,1]$ and observing that both parameters in the Beta function belong to the interval $(0,1]$, we have
			\begin{equation*}
				B\left(\frac{2}{p}, 2(p-2)\right) \leq \frac{1}{(p-2)(4-p)}.
			\end{equation*}
		\textbf{(2)} Case $\frac{5}{2} < p < 4$. Now $2(p-2) > 1$ and $\frac{2}{p} < 1$. Hence by Lemma \ref{11}, we have
			\begin{equation*}
				\begin{split}
					B\left(\frac{2}{p}, 2(p-2)\right) &< \frac{1}{\frac{2}{p}\cdot 2(p-2)} \left(\frac{2}{p} + 2(p-2) - \frac{2}{p}\cdot 2(p-2)\right)\\
										&= \frac{1}{2(p-2)}(p^2-4p+5).
				\end{split}
			\end{equation*}
			The claim follows from the following inequality 
			\begin{equation*}
				\frac{1}{2}(p^2-4p+5) \leq \frac{1}{4-p},
			\end{equation*}
			which is equivalent to
			\begin{equation*}
				(p-3)^2(p-2) \geq 0.
			\end{equation*}
\end{proof}

A crucial estimate for the proof of the main result \cite[Theorem 1.1]{BK} is Lemma 2.6 in \cite{BK}; see Lemma \ref{le26} below. Its proof is divided into two parts: $(i) \, 2+\sqrt{3} \le p < 4$ and $(ii) \, 2 < p < 2 + \sqrt{3}$. We provide a new proof for the case $(i)$. The proof of the case $(ii)$ is the same as in \cite{BK} and therefore omitted. It essentially uses Lemma \ref{le25} above and it is briefly commented at the end of the proof.

\begin{lem} 
\label{le26}
Let $2 < p < 4$ and $s \in [0,1]$. Then
	\begin{equation*}
		F_p(s) = \left(\frac{4-p}{2} + \frac{p-2}{2}s^4\right)B\left(\frac{2}{p}, 1-\frac{2}{p}\right) - \int_0^1 \psi_p(t)\max\{s^2,t^2\}^{p-2}dt \leq 0,
	\end{equation*}
	where $\psi_p(t) = t^{\frac{2}{p} -1}(1-t)^{-\frac{2}{p}}$.
\end{lem}
\begin{proof}[Proof.]
	First notice that
		\begin{equation*}
			F_p(1) = \left(\frac{4-p}{2} + \frac{p-2}{2}\right)B\left(\frac{2}{p}, 1 - \frac{2}{p}\right) - \int_0^1\psi_p(t)dt = 0.
		\end{equation*}
	The function $F_p(s)$ can be written as
		\begin{equation*}
			F_p(s) = \left(\frac{4-p}{2} + \frac{p-2}{2}s^4\right)B\left(\frac{2}{p}, 1-\frac{2}{p}\right) -s^{2(p-2)}\int_0^s\psi_p(t)dt - \int_s^1\psi_p(t)t^{2(p-2)}dt.
		\end{equation*}
	By differentiating $F_p$ we get
		\begin{equation*}
			\begin{split}
				F_p'(s) &= 2(p-2)s^3B\left(\frac{2}{p}, 1 - \frac{2}{p}\right) -2(p-2)s^{2p-5}\int_0^s\psi_p(t)dt -s^{2(p-2)}\psi_p(s) + s^{2(p-2)}\psi_p(s)\\
					&= 2(p-2)s^{2p-5}\left(B\left(\frac{2}{p}, 1 - \frac{2}{p}\right)s^{8-2p} -\int_0^s \psi_p(t)dt\right).
			\end{split}
		\end{equation*}
	Now a change of variable yields
		\begin{equation*}
			B\left(\frac{2}{p}, 1 - \frac{2}{p}\right) = \int_0^1 t^{\frac{2}{p} -1}(1-t)^{-\frac{2}{p}}dt = \int_0^s t^{\frac{2}{p} -1}(s-t)^{-\frac{2}{p}}dt.
		\end{equation*}
	Consequently, it follows that
		\begin{equation*}
			\begin{split}
				F_p'(s) &= 2(p-2)s^{2p-5}\left[s^{8-2p}\int_0^s t^{\frac{2}{p} -1}(s-t)^{-\frac{2}{p}}dt - \int_0^s t^{\frac{2}{p} -1}(1-t)^{-\frac{2}{p}}dt\right]\\
					&= 2(p-2)s^{2p-5}\left[\int_0^s t^{\frac{2}{p} -1}\left(s^{8-2p}(s-t)^{-\frac{2}{p}} - (1-t)^{-\frac{2}{p}}\right)dt\right].
			\end{split}
		\end{equation*}
	For fixed $p \in (2,4)$ and $s \in [0,1)$ define the function $H_{p, s}(t) = s^{8-2p}(s-t)^{-\frac{2}{p}} - (1-t)^{-\frac{2}{p}}$ for $t \in [0,s)$.

		\textbf{(1)} Case $2+\sqrt{3} \leq p < 4$.
			Solving $H_{p,s}(t) = 0$ we get
				\begin{equation*}
					s^{8-2p}(s-t)^{-\frac{2}{p}} = (1-t)^{-\frac{2}{p}} \iff s^{(4-p)p}(1-t) = (s-t).			
				\end{equation*}
			By denoting $x = (4-p)p$ and solving the equation above with respect to $t$ we have $t = \frac{s^x-s}{s^x-1}$. Since $s^x -1 < 0$ for all $s \in [0,1)$ solutions to $H_{p,s}(t) = 0$ exist only if there is some $t_0$ such that
				\begin{equation*}
					\begin{split}
						t_0 = \frac{s^x -s}{s^x-1} \geq 0 &\iff s^x - s \leq 0\\
											&\iff p^2-4p+1 \leq 0\\
											&\iff 2 - \sqrt{3} \leq p \leq 2 + \sqrt{3}.
					\end{split}
				\end{equation*}
			It also holds that $t_0 \leq s$ as
				\begin{equation*}
					t_0 \leq s \iff \frac{s^x-s}{s^x-1} \leq s \iff s < 1.
				\end{equation*}
			Because $ 2 + \sqrt{3} < p < 4$, it follows that $H_{p,s}(t) = 0$ has no solution in $[0,s)$. Moreover, $\lim_{t \to s} H_{p,s}(t) = +\infty$ and hence, $F_p'(s) \geq 0$. Thus $F_p(s)$ is nondecreasing in $[0,1]$ and because $F_p(1) = 0$ the statement holds.
		
		\textbf{(2)} Case $2 < p < 2 + \sqrt{3}$. The proof of this case essentially uses the estimate $F_p(0) \le 0,$ which follows from Lemma \ref{le25}. See the proof of \cite[Lemma 2.6]{BK} for details. 
\end{proof}

The main result of \cite{BK} is the following.
\begin{thm}[Theorem 1.1 in \cite{BK}]
\label{thm11}
Let $2 < p < 4$. Then the norm of the Hilbert matrix operator $\mathcal{H}$ acting on $A^p$ is $$\|\mathcal{H}\|_{A^p \to A^p} = \frac{\pi}{\sin(2\pi/p)}.$$
\end{thm}
The starting point of its proof is the representation of the Hilbert matrix operator in terms of weighted composition operators; see \eqref{eq: weighcomp}. The proof relies on a method based on Lemma \ref{le26} and using monotonicity of integral means in a new way yielding the upper estimate $$\|\mathcal{H}\|_{A^p \to A^p} \le \frac{\pi}{\sin(2\pi/p)},$$ which in combination with the same lower estimate 
implies Theorem \ref{thm11}. 
We refer the reader to \cite{BK} for the complete proof.  

\section{The norm of the Hilbert matrix operator on $H^\infty_{\alpha}$}

In this section, we derive norm estimates for the Hilbert matrix operator acting on $H^\infty_\alpha.$ Since $\mathcal{H}$ is not bounded on $H^\infty_\alpha$ for $\alpha = 0$ and $\alpha  \ge 1,$ we focus on the scale $0 < \alpha <1.$

\noindent Let $0 < \alpha < 1$ and $z \in \mathbb{D}$. Define
\begin{equation*}
	f_\alpha(z) = \frac{1}{(1-z)^\alpha}.
\end{equation*}
On one hand, we have the estimate 
$$\lVert f_\alpha \rVert_{H_{\alpha}^\infty} = \sup_{z \in \mathbb{D}} \frac{(1 - |z|^2)^\alpha}{|1-z|^\alpha} \leq \sup_{z \in \mathbb{D}} (1+|z|)^\alpha = 2^\alpha.$$ On the other hand, for $r \in (0,1)$ it holds that $\lim_{r \to 1-}|f_\alpha(r)| (1-r^2)^\alpha= 2^\alpha$ and we obtain $\lVert f_\alpha \rVert_{H_{\alpha}^\infty} = 2^\alpha.$ 

The weighted composition operator $T_t$ applied to a function $f_\alpha$ can be written as
\begin{equation*}
	T_t(f_\alpha)(z) = \frac{{((t-1)z +1)}^{\alpha -1}}{(1-t)^\alpha}f_{\alpha}(z), \hspace{0,25cm} z \in \mathbb{D},
\end{equation*}
and furthermore
\begin{equation*}
	\mathcal{H}(f_\alpha)(z) = \left(\int_0^1 \frac{((t-1)z +1)^{\alpha - 1}}{(1-t)^\alpha}dt\right)f_\alpha(z).
\end{equation*}
By a change of variables $t = 1 - s$ we can write $H(f_\alpha)(z) = \psi_\alpha(z) f_\alpha(z)$, where
\begin{equation*}
	\psi_\alpha(z) = \int_0^1 \frac{(1-sz)^{\alpha -1}}{s^\alpha}ds.
\end{equation*}
Notice that $\psi_\alpha(z)$ is defined and continuous on $\overline{\mathbb{D}}$ for $0 < \alpha < 1$. 
Thus we have for all $z \in \mathbb{D}$ and $0 < \alpha < 1$ that
$$|\psi_\alpha(z)| \leq  \int_0^1 \frac{(1-s)^{\alpha -1}}{s^\alpha}ds = \frac{\pi}{\sin(\alpha \pi)}.$$
We will also need the following representation of the norm of a weighted composition operator $uC_\varphi\colon H^\infty_\alpha \to H^\infty_\alpha:$ 
\begin{equation}
\label{eq: weightednorm}
\|u C_\varphi\|_{H^\infty_\alpha \to H^\infty_\alpha} = \sup_{z \in \D}\bigg(|u(z)|\bigg(\frac{1-|z|^2}{1-|\varphi(z)|^2}\bigg)^\alpha\bigg),
\end{equation}
where $u \in H^\infty$ and $\varphi$ is an analytic self-map of the unit disc; see \cite{CH}. 
We are now ready to give the lower bound for $\mathcal{H}$ in the $H_{\alpha}^\infty$ case.
\begin{thm}
\label{lowerbdd}
	If $\mathcal{H}: H_{\alpha}^\infty \to H_{\alpha}^\infty$, then 
	\begin{equation*}
		\lVert \mathcal{H} \rVert_{H_\alpha^\infty \to H_\alpha^\infty} \geq \frac{\pi}{\sin(\pi \alpha)}
	\end{equation*}
	holds for $0 < \alpha < 1$.
\end{thm}
\begin{proof}
	We will use the normalized test function $g_\alpha(z) = \frac{f_\alpha(z)}{\lVert f_\alpha \rVert_{H^\infty_\alpha}}$. We have
	\begin{equation*}
		\begin{split}
			\lVert \mathcal{H} \rVert_{H_\alpha^\infty \to H_\alpha^\infty} &\geq \lVert \mathcal{H}(g_\alpha) \rVert_{H_{\alpha}^\infty} = \lVert \psi_\alpha g_\alpha \lVert_{H_{\alpha}^\infty}
													= \sup_{z \in \mathbb{D}} |\psi_\alpha(z)||g_\alpha(z)|(1 - |z|^2)^\alpha\\
													&\geq \sup_{0 \leq r \leq 1} |\psi_\alpha(r)||g_\alpha(r)|(1- r^2)^\alpha
													= \sup_{0 \leq r \leq 1} \frac{|\psi_\alpha(r)|(1 + r)^\alpha}{\lVert f_\alpha \rVert_{H^\infty_\alpha}}
													= \frac{2^\alpha|\psi_\alpha(1)|}{\lVert f_\alpha \rVert_{H^\infty_\alpha}} \\
													&= |\psi_\alpha(1)| = \frac{\pi}{\sin(\pi \alpha)},
		\end{split}
	\end{equation*}
	 where we have used the equality $\lVert f_\alpha \rVert_{H^\infty_\alpha} = 2^\alpha$.
\end{proof}

Before proceeding to the proof of the upper bound for $\|\mathcal{H}\|_{H^\infty_\alpha \to H^\infty_\alpha}$ (Theorem \ref{upperbdd}), we establish the following lemma.

\begin{lem}
\label{normT_t}
Let $0 < \alpha < 1$. Then

$$\|T_t\|_{H^\infty_\alpha \to H^\infty_\alpha} = \begin{cases}
\frac{t^{\alpha-1}}{(1-t)^\alpha} \textup{ if } 0 < \alpha \le 2/3 \textup{ and } 0 < t < 1 \textup{ or if } 2/3 < \alpha < 1 \textup{ and } \frac{3\alpha - 2}{4\alpha - 2} \le t < 1 \\
(1-x_0)^{2\alpha-1}\bigg(\frac{1-|\frac{x_0}{1-t}|^2}{(1-x_0)^2-t^2}\bigg)^\alpha 
\textup{ if } 2/3 < \alpha < 1 \textup{ and } 0 < t < \frac{3\alpha - 2}{4\alpha - 2},
\end{cases}
$$ 
where 
$$x_0 = \frac{\alpha + 2\alpha t-t-\sqrt{4\alpha^2t-2\alpha t+\alpha^2-2\alpha+1}}{2\alpha-1}.$$
\end{lem}

\begin{proof}

Assume first that $0 < \alpha \leq 1/2$ and $f \in H^\infty_\alpha$. Then the estimate $|\phi_t'(z)| \le \frac{1-t}{t}$  and the Schwarz-Pick lemma yield 
	\begin{equation*}
		\begin{split}
			\lVert T_t(f) \rVert_{H_{\alpha}^\infty} &= \sup_{z \in \mathbb{D}} |T_t(f)(z)|(1-|z|^2)^\alpha \\
									&= \sup_{z \in \mathbb{D}} \frac{1}{\sqrt{t(1-t)}} |\phi_t'(z)|^{1/2}|f(\phi_t(z))|(1-|z|^2)^{\alpha}\\
									&=  \sup_{z \in \mathbb{D}} \frac{1}{\sqrt{t(1-t)}} |\phi_t'(z)|^{1/2-\alpha} |\phi_t'(z)|^{\alpha}|f(\phi_t(z))|(1-|z|^2)^{\alpha}\\
									&\leq \frac{1}{\sqrt{t(1-t)}} \left( \frac{1-t}{t}\right)^{1/2- \alpha} \sup_{z \in \mathbb{D}} |f(\phi_t(z))|(1-|z|^2)^{\alpha} |\phi_t'(z)|^{\alpha}\\
									&\le \frac{t^{\alpha-1}}{(1-t)^\alpha} \sup_{z \in \mathbb{D}} |f(\phi_t(z))|(1-|\phi_t(z)|^2)^{\alpha} \le  \frac{t^{\alpha-1}}{(1-t)^\alpha} \lVert f \rVert_{H_{\alpha}^\infty}.
		\end{split}
	\end{equation*}
	
Let $1/2 < \alpha < 1.$ 
We apply equation \eqref{eq: weightednorm} 
to $T_t$, in which case we have $$u(z) =\omega_t(z)= \frac{1}{(t-1)z+1} =  \frac{1}{\sqrt{t(1-t)}} (\phi_t'(z))^{1/2}$$ and $\varphi = \phi_t$ and thus 
we get $$\|T_t\|_{H^\infty_\alpha \to H^\infty_\alpha} = \sup_{z \in \D}|1-(1-t)z|^{2\alpha-1}\bigg(\frac{1-|z|^2}{|1-(1-t)z|^2-t^2}\bigg)^\alpha.$$ Define $$F(z) = |1-(1-t)z|^{2\alpha-1}\bigg(\frac{1-|z|^2}{|1-(1-t)z|^2-t^2}\bigg)^\alpha$$ and a projection $$R(z) = 1 - |z-1|$$ that rotates a point $z \in \D$ around the point 1 placing it on the real axis. We have that $|R(z)| \le |z|$ and $|1- R(z)| = |1-z|$ for each $z \in \D$, so it follows that $$F(z) \le |1-R((1-t)z)|^{2\alpha-1}\bigg(\frac{1-|\frac{R((1-t)z)}{1-t}|^2}{|1-R((1-t)z)|^2-t^2}\bigg)^\alpha.$$ 

Since $|R((1-t)z)| \le 1-t$, we have $$\sup_{z \in \D}F(z) \le \sup_{|x| \le 1-t}G(x),$$ where $$G(x) = (1-x)^{2\alpha-1}\bigg(\frac{1-|\frac{x}{1-t}|^2}{(1-x)^2-t^2}\bigg)^\alpha$$ is a differentiable function in $[t-1,1-t)$ with $\lim_{x \to 1-t}G(x) = \frac{t^{\alpha-1}}{(1-t)^\alpha}.$ Assuming that $x \ne 1-t$, we may write $$G(x) = (1-x)^{2\alpha-1}\bigg(\frac{1-t+x}{(1-t)^2(1+t-x)}\bigg)^\alpha.$$ Denoting its continuous extension to $[t-1,1-t]$ still by $G$ we have 
$$G(x) = (1-x)^{2\alpha-1}\bigg(\frac{1-t+x}{(1-t)^2(1+t-x)}\bigg)^\alpha$$ that is a differentiable function in $[t-1,1-t]$ and $G(1-t) = \frac{t^{\alpha-1}}{(1-t)^\alpha}$. We obtain \[
\begin{split}
G'(x) &=\alpha(1-x)^{2\alpha-1}\bigg(\frac{1-t+x}{(1-t)^2(1+t-x)^2}+\frac{1}{(1-t)^2(1+t-x)}\bigg)\bigg(\frac{1-t+x}{(1-t)^2(1+t-x)}\bigg)^{\alpha-1}
\\
&-(2\alpha-1)(1-x)^{2(\alpha-1)}\bigg(\frac{1-t+x}{(1-t)^2(1+t-x)}\bigg)^\alpha. 
\end{split}
\]
For $x \in (t-1,1-t]$, we may write $$G'(x) =\frac{(1-x)^{2(\alpha-1)}\bigg(\frac{1-t+x}{(t-1)^2(1+t-x)}\bigg)^\alpha((1-2\alpha)x^2+(4\alpha t-2t+2\alpha)x+(1-2\alpha)t^2-1)}{(t-x)^2-1}.$$ Since $x=t-1$ corresponds to the point $z=-1$ and $F(-1)=G(t-1) = 0$, we may omit it. The zeros of $G$ in $(t-1,1-t]$ are the zeros of a quadratic polynomial 
$$p(x)=(1-2\alpha)x^2+(4\alpha t-2t+2\alpha)x+(1-2\alpha)t^2-1.$$ From $p(x)=0,$ we get $$x_0 = \frac{\alpha + 2\alpha t-t-\sqrt{4\alpha^2t-2\alpha t+\alpha^2-2\alpha+1}}{2\alpha-1},$$ since the other root of $p(x)$ is located outside the interval $(t-1,1-t]$ and therefore it can be discarded. Looking at the sign of $G'(x)$ near $x_0,$ we observe that $G$ has a global maximum at $x_0>0$  whenever $x_0 \in [t-1,1-t]$. Let us first confirm that $x_0 > 0.$ Now \[
\begin{split}
(\alpha + 2\alpha t-t)^2 &= \alpha^2+4\alpha^2t + 4\alpha^2t^2-2t\alpha-4\alpha t^2+t^2 = \alpha^2+4\alpha^2t-2t\alpha+(2\alpha-1)^2t^2 
\\
&>  \alpha^2+4\alpha^2t-2t\alpha +1-2\alpha,
\end{split}
\]
since $1-2\alpha < 0.$ Hence $$x_0 = \frac{\alpha + 2\alpha t-t-\sqrt{4\alpha^2t-2\alpha t+\alpha^2-2\alpha+1}}{2\alpha-1} > \frac{\alpha + 2\alpha t-t-\sqrt{(\alpha+2\alpha t-t)^2}}{2\alpha-1}=0.$$

Let us next investigate when $x_0 \le 1-t.$ Now for $1/2 < \alpha < 1$, we have

\[
\begin{split}
x_0 &\le 1-t  \\
\iff \alpha + 2\alpha t-t-\sqrt{4\alpha^2t-2\alpha t+\alpha^2-2\alpha+1} &\le (2\alpha-1)(1-t) \\
\iff (4\alpha-2)t + 1 - \alpha &\le \sqrt{4\alpha^2t-2\alpha t+\alpha^2-2\alpha+1} \\
\iff (8\alpha^2-8\alpha+2)t^2+(7\alpha-6\alpha^2-2)t &\le 0 \\
\iff 8(\alpha-1/2)^2 t &\le 6\alpha^2-7\alpha+2 = 6(\alpha-1/2)(\alpha-2/3) \\
\iff t &\le \frac{3\alpha - 2}{4\alpha - 2}.
\end{split}
\]
Hence we have the following cases. 

\textbf{(1)} If $1/2 < \alpha \le 2/3,$ then $x_0 \le 1-t$ if and only if $t \le  \frac{3\alpha - 2}{4\alpha - 2} < 0$. Since $0 < t < 1,$ it holds that $x_0 > 1-t$ and $\|T_t\|_{H^\infty_\alpha \to H^\infty_\alpha} = \sup_{z \in \D}F(z) = \lim_{x\to 1}F(x) = G(1-t)=\frac{t^{\alpha-1}}{(1-t)^\alpha}$.

\textbf{(2)} If $2/3 < \alpha < 1,$ then $x_0 \le 1-t$ if and only if   $0 < t \le  \frac{3\alpha - 2}{4\alpha - 2}$. In this case, $$\|T_t\|_{H^\infty_\alpha \to H^\infty_\alpha} = F(x_0/(1-t))=G(x_0)=(1-x_0)^{2\alpha-1}\bigg(\frac{1-|\frac{x_0}{1-t}|^2}{(1-x_0)^2-t^2}\bigg)^\alpha.$$ Otherwise, we have $x_0 > 1-t$ if and only if $ \frac{3\alpha - 2}{4\alpha - 2} < t < 1$ and then $\|T_t\|_{H^\infty_\alpha \to H^\infty_\alpha} = G(1-t) = \frac{t^{\alpha-1}}{(1-t)^\alpha}$.

Now in combination with the case $0 < \alpha \le 1/2$, we have the following result. 

$$\|T_t\|_{H^\infty_\alpha \to H^\infty_\alpha} = \begin{cases}
\frac{t^{\alpha-1}}{(1-t)^\alpha} \textup{ if } 0 < \alpha \le 2/3 \textup{ and } 0 < t < 1 \textup{ or if } 2/3 < \alpha < 1 \textup{ and } \frac{3\alpha - 2}{4\alpha - 2} \le t < 1 \\
(1-x_0)^{2\alpha-1}\bigg(\frac{1-|\frac{x_0}{1-t}|^2}{(1-x_0)^2-t^2}\bigg)^\alpha 
\textup{ if } 2/3 < \alpha < 1 \textup{ and } 0 < t < \frac{3\alpha - 2}{4\alpha - 2},
\end{cases}
$$ 
where $$x_0 = \frac{\alpha + 2\alpha t-t-\sqrt{4\alpha^2t-2\alpha t+\alpha^2-2\alpha+1}}{2\alpha-1}.$$

\end{proof}

\begin{thm}
\label{upperbdd}

Let $0 < \alpha \le 2/3$ and $\mathcal{H}\colon H^\infty_\alpha \to H^\infty_\alpha$ be the Hilbert matrix operator. Then $$\|\mathcal{H}\|_{H^\infty_\alpha \to H^\infty_\alpha} = \frac{\pi}{\sin(\alpha \pi)}.$$ For $2/3 < \alpha < 1$, we have the following upper bound \[
\begin{split}
\|\mathcal{H}\|_{H^\infty_\alpha \to H^\infty_\alpha} \le \int_0^{\frac{3\alpha-2}{4\alpha-2}}G(x_0) \, dt + \int_{\frac{3\alpha-2}{4\alpha-2}}^1 t^{\alpha - 1}(1-t)^{-\alpha} \, dt,
\end{split}
\]
where $$G(x)= (1-x)^{2\alpha-1}\bigg(\frac{1-|\frac{x}{1-t}|^2}{(1-x)^2-t^2}\bigg)^\alpha$$ and $$x_0 = \frac{\alpha + 2\alpha t-t-\sqrt{4\alpha^2t-2\alpha t+\alpha^2-2\alpha+1}}{2\alpha-1}.$$ 
\end{thm}

\begin{proof}
Let $0 < \alpha \le 2/3$. Then by Lemma \ref{normT_t} we have for $f \in H^\infty_\alpha$ with $\|f\|_{H^\infty_\alpha} = 1$ that 
\[
\begin{split}
\|\mathcal{H}f\|_{H^\infty_\alpha} &= \sup_{z \in \D} |\int_0^1 T_t f (z)\, dt (1-|z|^2)^\alpha| \le \int_0^1 \sup_{z \in \D} |T_t f (z)|  (1-|z|^2)^\alpha \, dt 
\\
&\le \int_0^1 \|T_t\|_{H^\infty_\alpha \to H^\infty_\alpha} \, dt = \int_0^1 t^{\alpha - 1}(1-t)^{-\alpha} \, dt = B(\alpha, 1-\alpha) = \frac{\pi}{\sin(\alpha \pi)}. 
\end{split}
\]
Thus in view of the lower bound (Theorem \ref{lowerbdd}) we have established $$\|\mathcal{H}\|_{H^\infty_\alpha \to H^\infty_\alpha} = \frac{\pi}{\sin(\alpha \pi)}$$ for $0 < \alpha \le 2/3$. In the case $2/3 < \alpha < 1,$ we obtain the following upper bound: 
\[
\begin{split}
\|\mathcal{H}\|_{H^\infty_\alpha \to H^\infty_\alpha} &\le  \int_0^1 \|T_t\|_{H^\infty_\alpha \to H^\infty_\alpha} \, dt 
\\
&=\int_0^{\frac{3\alpha-2}{4\alpha-2}}G(x_0) \, dt + \int_{\frac{3\alpha-2}{4\alpha-2}}^1 t^{\alpha - 1}(1-t)^{-\alpha} \, dt,
\end{split}
\]
where $$G(x)= (1-x)^{2\alpha-1}\bigg(\frac{1-|\frac{x}{1-t}|^2}{(1-x)^2-t^2}\bigg)^\alpha$$ and $$x_0 = \frac{\alpha + 2\alpha t-t-\sqrt{4\alpha^2t-2\alpha t+\alpha^2-2\alpha+1}}{2\alpha-1}$$ are from Lemma \ref{normT_t}.
\end{proof}

\end{document}